\documentclass[11pt]{article}
\usepackage{amssymb,amsmath,amsthm,bm}
\usepackage{verbatim}
\usepackage[colorinlistoftodos,prependcaption,textsize=tiny]{todonotes}
 \definecolor{darkblue}{RGB}{0,0,160}
\usepackage{fouriernc} 
\usepackage[colorlinks=true,allcolors=darkblue]{hyperref}
\DeclareSymbolFont{usualmathcal}{OMS}{cmsy}{m}{n}
\DeclareSymbolFontAlphabet{\mathcal}{usualmathcal}

\usepackage[T1]{fontenc}

\usepackage{color}

\usepackage{parskip}
\usepackage{setspace}

\usepackage[capitalise]{cleveref}

\newtheorem{theorem}{Theorem}[section]

\newtheorem{lemma}[theorem]{Lemma}

\newtheorem*{definition*}{Definition}

\def\thang#1{\noindent
\textcolor{blue}
{\textsc{(Thang:}
\textsf{#1})}}

\def\paige#1{\noindent
\textcolor{purple}
{\textsc{(Paige:}
\textsf{#1})}}

\def\F{\mathcal{F}}

\def \F{{\mathbb F}}

\def\F{\mathbb{F}}

\begin{document}

\title{On a radial projection conjecture and pinned directions \\ in finite spaces}
\author{Paige Bright \and Ben Lund  \and Thang Pham }
\date{}

\maketitle

\begin{abstract}
    We give upper bounds on the number of exceptional radial projections of arbitrary subsets of vector spaces over finite fields.
    Our bounds do not depend on the dimension of the ambient space.

    Let $\F_q^d$ be the $d$-dimensional vector space over $\F_q$, let $k \in \{1,2,\ldots,d-1\}$, and let $E \subseteq \F_q^d$ be an arbitrary set of points. 
    We prove two results.
    First, if $q^{k-1} < |E| \leq 100^{-1}q^{k}$, then the number of points $y$ such that the projection of $E$ from $y$ contains fewer than $50^{-1}|E|$ points is bounded above by $40q^k$.
    This establishes a conjecture of Lund, Pham, and Thu.
    Second, if $30q^{k} \leq |E| \leq q^{k+1}$, then the number of points $y$ such that the projection of $E$ from $y$ contains fewer than $M \leq 4^{-1}q^k$ points is bounded above by $300q^kM|E|^{-1}$. 
    We also have an application to a pinned directions problem.
    Specifically, if $E\subset \mathbb{F}_q^d$ with $|E| > 30q^k$, then there is a point $y \in E$ such that the set of lines incident to $y$ and at least one other point of $E$ determines $q^k/4$ distinct slopes.

\end{abstract}

\section{Introduction}
For any point $y \in \F^d$, the {\em radial projection} $\pi^y$ maps each point $x \in \F^d \setminus \{y\}$ to the line that contains $x$ and $y$.
For $E \subset \F^d$, we denote $\pi^y(E) = \bigcup_{x \in E \setminus \{y\}}\pi^y(x)$.

For $\F = \F_q$, the finite field with $q$ elements, there are two obvious upper bounds on $|\pi^y(E)|$.
First, $|\pi^y(E)| \leq |E|$.
Second, $|\pi^y(E)|$ is at most the number of lines through $y$, which is $\frac{q^d-1}{q-1}$.
More generally, if $E$ is contained in an affine $k$-plane $\Gamma$, then $|\pi^y(E)| \leq \frac{q^k-1}{q-1}$ for each point $y \in \Gamma$.

A projection $\pi^y$ is {\em exceptional} for a set $E$ if $\pi^y(E)$ is much smaller than these trivial upper bounds.
In this paper, we bound the number of exceptional projections for arbitrary sets $E \subseteq \F_q^d$.


Our first result concerns the number of projections whose image of $E$ is  smaller than $|E|$ by at least a constant factor.
\begin{theorem}\label{conj:ff}
Let $k \in \{1,2,\ldots,d-1\}$. Let $E \subset \mathbb{F}_q^d$ with $q^{k-1} \leq |E| \leq 100^{-1}q^k$.
Then,
\begin{equation}\label{eq:conjFF}\#\{y \in \mathbb{F}_q^d: |\pi^y(E)| \leq 50^{-1}|E|\} \leq 40q^k . \end{equation}
\end{theorem}
\cref{conj:ff} was conjectured by Lund, Pham, and Thu \cite[Conjecture 1.6]{LPT}, and they proved the case $k=d-1$ \cite[Theorem 1.4]{LPT}.
As mentioned above, if $E$ is contained in an affine $k$-plane $\Gamma$, then, for any $y \in \Gamma$, we have that $|\pi^y(E)| \leq \frac{q^{k} - 1}{q-1} < 2|E|$.
Consequently, \cref{conj:ff} is tight up to constant factors for each fixed $k$.
In fact, this example shows more.
Even if we replace the bound $50^{-1}|E|$ on the left hand side of \cref{eq:conjFF} by $\frac{q^k-1}{q-1}$, we still cannot hope to reduce the $40q^{k}$ on the right hand side by more than a constant factor.

However, we do obtain an improved bound on the number of projections of $E$ whose image is smaller than $q^{k}$ in the case that the size of $E$ is so large that it cannot be contained in a $k$-plane.
The case $k=d-1$ of \cref{th:fullDimLargeESC} was proved by Lund, Pham, and Thu \cite[Theorem 1.1]{LPT}.

\begin{theorem}\label{th:fullDimLargeESC}
Let $k \in \{1,2,\ldots,d-1\}$.
Let $E \subset \mathbb{F}_q^d$, with $30q^k \leq |E| \leq q^{k+1}$, and let $M$ be an integer with $|E|q^{-1} \leq M \leq 4^{-1} q^{k}$.
Then,
\begin{equation} \label{eq:fullDimLargeESC}
\#\{y\in \mathbb{F}_q^d\colon |\pi^y(E)|\leq M\} <  300q^{k}M|E|^{-1}.
\end{equation}
\end{theorem}

Since each line contains at most $q$ points, we have the trivial lower bound $|\pi^y(E)| \geq |E|q^{-1}$.
Hence, for $M < |E|q^{-1}$, we obtain the equality
$\{y \in \F_q^d: |\pi^y(E)| \leq M\} = 0.$

\subsection{An application to pinned directions}
Another perspective on this study, as mentioned in \cite{LPT}, is that it gives a pinned direction set result. In the plane over the reals, Ungar \cite{Ungar}  proved that for any finite set $E$ of points in $\mathbb{R}^2$, we have either $E$ is contained in a single line, or $E$ determines at least $|E|-1$ distinct directions. The pinned version of this result was given by Beck in \cite{Beck}, namely, he showed that either most points of $E$ are contained in a single line, or there is a point $x \in E$ such that $|\pi^x(E)| \gg |E|$. When we replace $\mathbb{R}$ by $\mathbb{F}_p$, where $p$ is a prime, Sz\H{o}nyi \cite{Szonyi} proved that if $E \subset \mathbb{F}_p^2$ is a set of points not contained in a line and $|E| \leq p$, then $E$ determines at least $(|E|+3)/2$ directions. The results in this paper can be viewed as pinned versions of this topic in all dimensions and over arbitrary finite fields. In particular, we include the following application of Theorem \ref{th:fullDimLargeESC}.
\begin{theorem}
Let $k\in \{1, 2, \ldots, d-1\}$ and let $E\subset \mathbb{F}_q^d$. 
If $|E| > 30q^k$, then there is a point $y \in E$ such that $|\pi^y(E)| \ge q^k/4$.
\end{theorem}
\begin{proof}
    Let $T = \{z \in \F_q^d: |\pi^z(E)| \leq q^k/4\}$.
    By \cref{th:fullDimLargeESC}, we have that $|T| \leq 75q^{2k}|E|^{-1} < |E|$.
    Hence, $E \setminus T$ is nonempty.
\end{proof}
\subsection{Main ideas}
The proofs of Theorems~\ref{conj:ff} and~\ref{th:fullDimLargeESC} share a common strategy that distinguishes our approach from the incidence-based arguments developed in~\cite{LPT}: we dimension-reduce the problem by finding a favorable low-dimensional projection that approximately preserves the sizes of both the point set and the exceptional set.

The key technical ingredient is Lemma~\ref{th:randomOrthogonalProjection}, which guarantees that for any two sets $A,B \subset \mathbb{F}_q^d$ with $|A|,|B| < q^{k+1}$, there exists a $(d-k-1)$-dimensional subspace $\Gamma$ such that the linear projection
\[
\pi_\Gamma : \mathbb{F}_q^d \to \mathbb{F}_q^d / \Gamma \cong \mathbb{F}_q^{k+1}
\]
preserves at least a constant fraction of the cardinalities of both sets:
\[
|\pi_\Gamma(A)| \ge \tfrac15 |A|
\quad\text{and}\quad
|\pi_\Gamma(B)| \ge \tfrac15 |B|.
\]
The proof uses a probabilistic averaging and second-moment argument. In particular, we count pairs of points that collide under projection in two ways: first by averaging over subspaces, which controls the expected number of collisions, and then by applying Markov’s inequality and Cauchy–Schwarz to show that a positive proportion of subspaces avoid excessive collisions, leading to the desired lower bounds for $|\pi_\Gamma(A)|$ and $|\pi_\Gamma(B)|$.

The main step in the proof of Theorem~\ref{th:fullDimLargeESC} is to apply Lemma~\ref{th:randomOrthogonalProjection} simultaneously to $E$ and to the exceptional set
\[
T = \{ y \in \mathbb{F}_q^d : |\pi^y(E)| \le M \}.
\]
At first, we do not know if $|T|$ is small enough to apply Lemma~\ref{th:randomOrthogonalProjection}. This difficulty is overcome using Lemma~\ref{th:weakProjBound}, taken from~\cite{LPT}, which provides a crude but sufficient upper bound
\[
|T| < (C-1)^{-1} q |E|,
\]
whenever $M < C^{-1} |E|$. Under the hypotheses of Theorem~\ref{th:fullDimLargeESC}, this bound ensures $|T| < q^{k+1}$, thereby placing both $E$ and $T$ in the range where Lemma~\ref{th:randomOrthogonalProjection} applies.

Once we have found a favorable subspace $\Gamma$, the following key geometric observation allows us to relate exceptional projections in $\mathbb{F}_q^d$ to exceptional projections in $\mathbb{F}_q^{k+1}$. For each $y \in T$, let $L_y$ be the set of lines through $y$ that intersect $E$. Under the projection $\pi_\Gamma$, each line in $L_y$ either maps to a line through $\pi_\Gamma(y)$ that intersects $\pi_\Gamma(E)$ or collapses to the point $\pi_\Gamma(y)$ itself. In particular, the number of distinct lines through $\pi_\Gamma(y)$ intersecting $\pi_\Gamma(E)$ is at most
\[
|L_y| = |\pi^y(E)| \le M.
\]
Hence
\[
\pi_\Gamma(T) \subseteq \bigl\{ w \in \mathbb{F}_q^{k+1} : |\pi^w(\pi_\Gamma(E))| \le M \bigr\}.
\]
This reduces the problem to the $(k+1)$-dimensional space $\mathbb{F}_q^{k+1}$, where Theorem~\ref{th:largeESC}, taken from~\cite{LPT} and now applied in dimension $k+1$, provides the desired sharp bound on the number of exceptional points. Comparing this upper bound for $|\pi_\Gamma(T)|$ with the lower bound $|\pi_\Gamma(T)| \ge \tfrac15 |T|$ given by Lemma~\ref{th:randomOrthogonalProjection} yields the estimate asserted in Theorem~\ref{th:fullDimLargeESC}.

The proof of Theorem~\ref{conj:ff} follows the same blueprint, but uses \cref{th:smallESC} (also taken from \cite{LPT}) in place of \cref{th:largeESC}.

Overall, the dimension-reduction mechanism provided by Lemma~\ref{th:randomOrthogonalProjection} allows us to lift the codimension-$1$ estimates of~\cite{LPT} to arbitrary codimensions without introducing any dependence on the ambient dimension $d$ in our bounds. 

The continuous analogs of \cref{conj:ff} and \cref{th:fullDimLargeESC} were recently proved in two papers, by Bright and Gan \cite{BG}, and by Orponen, Shmerkin, and Wang \cite{OSW}.

\subsection{Discussions}
In this subsection, we emphasize a connection between the topic studied in this paper and the Erd\H{o}s-Falconer distance problem, which asks for the smallest exponent $\alpha$ such that for any $E\subset \mathbb{F}_q^d$, if $|E|\ge q^{\alpha}$, then the distance set covers a positive proportion of all distances. This is the finite field analog of the celebrated Falconer distance conjecture in the continuous setting, stating that if the Hausdorff dimension of a compact set $E\subset \mathbb{R}^d$ is greater than $d/2$, then its distance set has positive Lebesgue measure.

In the finite field setting, Iosevich and Rudnev \cite{IR07} initially showed that if $|E|\gg q^{\frac{d+1}{2}}$, then $E$ determines a positive proportion of all distances. The exponent $\frac{d+1}{2}$ has been shown to be optimal in \cite{HIKR10} in odd dimensions. In even dimensions, the conjectured exponent is $d/2$, which is directly in line with the Falconer distance conjecture. The best current exponents in the literature for the plane case are $\frac{5}{4}$ and $\frac{4}{3}$ due to Murphy, Petridis, Pham, Rudnev, and Stevens \cite{murphy} and Chapman, Erdogan, Hart, Iosevich, and Koh \cite{CEHIK10}, respectively. In higher even dimensions, no exponent better than \(\frac{d+1}{2}\) is currently known. However, in the continuous setting, the whole picture is much clearer. For example, Du, Iosevich, Ou, Wang, and Zhang \cite{Du3} proved that for any compact set $E$ in $\mathbb{R}^d$, $d\ge 4$ even, with the Hausdorff dimension of at least $\frac{d}{2}+\frac{1}{4}$, its distance set has positive Lebesgue measure. One of the key steps in their paper is a radial projection theorem developed in \cite{o2}, thus, extending these methods from $\mathbb{R}^d$ to $\mathbb{F}_q^d$ requires a thorough understanding of radial projections over finite fields. 
In other words, our results contribute further to this line of research.

\section{Notation, terminology, and basic facts about finite geometry}

The number of $k$-dimensional linear subspaces of $\mathbb{F}_q^d$ is
\[ \binom{d}{k}_q = \frac{(q^d-1)(q^d-q)\ldots(q^d-q^{k-1})}{(q^k-1)(q^k-q)\ldots(q^k-q^{k-1})}.\]
Indeed, the numerator counts the number of ordered $k$-tuples of  linearly independent vectors in $\F_q^d$, and the denominator counts the number of ordered $k$-tuples of linearly independent vectors in $\F_q^k$.
Similarly, the number of $k$-dimensional subspaces that contain a fixed $\ell$-dimensional subspace (for $\ell \leq k$) is $\binom{d-\ell}{k-\ell}_q$.

For a linear subspace $\Gamma$ in $ \F_q^d$ and a point $x \in \F_q^d$, denote
$\pi_\Gamma(x) = \Gamma + x$.
For a set of points $S \subseteq \F_q^d$, denote
$\pi_\Gamma(S) = \bigcup_{x \in S} \pi_\Gamma(x)$.
In our application, $\Gamma$ will be $(d-k-1)$-dimensional.
Under this assumption, the space $\pi_\Gamma(\F_q^d)$ of translates of $\Gamma$ is a $k+1$-dimensional vector space over $\F_q$ via $\mathbb{F}_q^d/\Gamma\cong \mathbb{F}_q^{k+1}$, and we consider the projections $\pi_\Gamma(x)$ of points in $\F_q^d$ to be elements of $\F_q^{k+1}$.
For $y \in \F_q^{k+1}$, we denote $\pi_\Gamma^{-1}(y)$ to be the set of points $x$ such that $\pi_\Gamma(x) = y$; note that $\pi_\Gamma^{-1}(y)$ is the set of points in an affine $(d-k-1)$-plane.
Similarly, for a set $S \subseteq \F_q^{k+1}$, we denote $\pi_\Gamma^{-1}(S) = \bigcup_{y \in S} \pi_\Gamma^{-1}(y)$.

\section{Proof of \cref{conj:ff,th:fullDimLargeESC}}
To prove our theorems, we first find a $(d-k-1)$-dimensional subspace $\Gamma$ such that $\pi_\Gamma$ nearly preserves the size of $E$ and the size of $T = \{y \in \F_q^d: |\pi^y(E)| \le M\}$.



We show that if neither $|E|$ nor $|T|$ is too large, then there is a projection that roughly preserves the size of both sets. The following lemma is of independent interest, and is expected to have more applications in other topics.

\begin{lemma}\label{th:randomOrthogonalProjection}
    Let $A,B \subset \F_q^d$, with $|A| < q^{k+1}$ and $|B| < q^{k+1}$. 
    Then, there exists an $(d-k-1)$-dimensional subspace $\Gamma$ such that $|\pi_\Gamma(A)| \geq 5^{-1}|A|$ and $|\pi_\Gamma(B)| \geq 5^{-1}|B|$.
\end{lemma}
\begin{proof}
    Let $X \subset \F_q^d$, with $|X| < q^{k+1}$.
    Let $R$ be the set of triples $(\Lambda,u,v)$ where $\Lambda$ is an affine $(d-k-1)$-plane and $u$ and $v$ are distinct points of $X$ that are contained in $\Lambda$.
    We count $|R|$ in two different ways.
    First, we take the sum over $(d-k-1)$-planes $\Lambda$ of the number of pairs of points of $X$ contained in $\Lambda$.
    Second, we take the sum over pairs of points $P$ of $X$ of the number of $(d-k-1)$-planes that contain $P$.
    In detail,
    \[|R| = \sum_{\Gamma \in G(d,d-k-1)} \sum_{y \in \F_q^{k+1}} \binom{|X \cap \pi_\Gamma^{-1}(y)|}{2} = \binom{d-1}{d-k-2}_q \binom{|X|}{2}. \]
    Note that
    \[ |G(d,d-k-1)| = \binom{d}{d-k-1}_q = \frac{q^{d}-1}{q^{k+1}-1}\binom{d-1}{d-k-1}_q = \frac{q^{d}-1}{q^{d-k-1}-1}\binom{d-1}{d-k-2}_q. \]
    Hence, the expected number of pairs of distinct points of $X$ with the same image under $\pi_\Gamma$ for a uniformly random $\Gamma$ is
    \[E_{\Gamma}\left(\sum_{y \in \F_q^{k+1}} \binom{|X \cap \pi^{-1}_\Gamma(y)|}{2}\right) = |R|\,|G(d,d-k-1)|^{-1} = \frac{q^{d-k-1}-1}{q^d-1}\binom{|X|}{2}.\]
    By Markov's inequality, the following condition is satisfied for at least $(3/4)|G(d,d-k-1)|$ of the $(d-k-1)$-dimensional subspaces $\Gamma$:
    \begin{equation}\label{eq:projectionCondition}
    \sum_{y \in \F_q^{k+1}} \binom{|X \cap \pi^{-1}_\Gamma(y)|}{2} \leq 4  \cdot \frac{q^{d-k-1}-1}{q^d-1}\binom{|X|}{2}.
    \end{equation}
    In particular, \cref{eq:projectionCondition} is simultaneously satisfied for $X=A$ and $X=B$ for at least half of the $(d-k-1)$-dimensional subspaces $\Gamma$.

    Now suppose that \cref{eq:projectionCondition} is satisfied for a fixed set $X$ and $(d-k-1)$-dimensional subspace $\Gamma$, and for each $y \in \F_q^{k+1}$, denote $X(y) = |X \cap \pi^{-1}_\Gamma(y)|$.
    Using the assumption that $|X| \leq q^{k+1}$,
    \[\sum_{y \in \F_q^{k+1}} X(y)^2 = 2 \sum_{y \in \F_q^{k+1}} \binom{X(y)}{2} + \sum_{y \in \F_q^{k+1}} X(y) < 8q^{-k-1} \binom{|X|}{2} + |X| \leq 5|X|. \]
    Combining this with the Cauchy-Schwarz inequality, we have
    \[|X|^2 = \left(\sum_{y \in \F_q^{k+1}} X(y) \right)^2 \leq |\pi_\Gamma(X)|\sum_{y \in \Gamma^\perp}X(y)^2 \leq 5|\pi_\Gamma(X)|\,|X|. \]
    Hence, $|\pi_\Gamma(X)| \geq 5^{-1}|X|$.
    As noted above, there is a choice of $\Gamma$ so that \cref{eq:projectionCondition} is simultaneously satisfied for both $A$ and $B$.
\end{proof}

In order to apply \cref{th:randomOrthogonalProjection} to $T$, we need to already know that $|T|$ is not too large.
The following lemma taken from \cite[Theorem 1.5]{LPT} is sufficient for this purpose.

\begin{lemma}[\cite{LPT}]\label{th:weakProjBound}
Let $E \subset \F_q^d$ and let $1 < C < |E|$. Then,
\[\#\{y \in \F_q^d : |\pi^y(E)| < C^{-1}|E| \} < (C-1)^{-1}q|E|. \]
\end{lemma}
The proof of \cref{th:fullDimLargeESC} also makes use of the result for the case $k=d-1$, which was proved in  \cite[Theorem 1.1]{LPT}.
\begin{theorem}[\cite{LPT}]\label{th:largeESC}
Let $E \subset \mathbb{F}_q^d$ and let $M$ be a positive integer, with $|E| \geq 6q^{d-1}$ and $M \leq 4^{-1}q^{d-1}$.
Then,
\begin{equation} \label{eq:largeESC}
\#\{y\in \mathbb{F}_q^d\colon |\pi^y(E)|\leq M\} <  12q^{d-1}M|E|^{-1}.\end{equation}
\end{theorem}
We are now ready to prove \cref{th:fullDimLargeESC}.

\begin{proof}[Proof of \cref{th:fullDimLargeESC}]
Denote $|E| = cq^k \geq 30q^k$, and
\[T := \{y \in \F_q^d: |\pi^y(E)| \le M\}. \]
Since $M < 4^{-1}q^{k} = (4c)^{-1}|E|$, \cref{th:weakProjBound} implies that $|T| \leq (4c-1)^{-1}q|E| < 3^{-1}q^{k+1}$.
Since both $|E| \leq q^{k+1}$ and $|T| \leq q^{k+1}$, \cref{th:randomOrthogonalProjection} implies that there is an $(d-k-1)$-dimensional subspace $\Gamma$ such that $\pi_\Gamma(T) \geq 5^{-1}|T|$ and $\pi_\Gamma(E) \geq 5^{-1}|E|$.

For $y \in T$, let $L_y$ be the set of lines that contain $y$ and at least one point of $E$, and let $L_y^\pi$ be the set of lines that contain $\pi_\Gamma(y)$ and at least one point of $\pi_\Gamma(E)$.
Note that $|L_y| \geq |L^\pi_y|$.
Indeed, the image of each line $\ell \in L_y$ under $\pi_\Gamma$ is either a line in $L^\pi_y$ or the point $\pi(y)$.
Hence, $|\pi^{\pi_\Gamma(y)}(\pi_\Gamma(E))| = |L^\pi_y| \leq |L_y| \leq M$, and so 
\[\pi_\Gamma(T) \subseteq \{w \in \F_q^{k+1} : |\pi^w(\pi_\Gamma(E))| \leq M\}.\]
Since $|\pi_{\Gamma}(E)| \geq 5^{-1}|E| \geq 6q^k$ and $M \leq 4^{-1}q^{k}$, we can apply \cref{th:largeESC} to obtain  \[5^{-1}|T| \leq |\pi_\Gamma(T)| \leq \#\{w \in \F_q^{k+1}: |\pi^w(\pi_\Gamma(E))| \leq M\} \leq 12q^kM|\pi_\Gamma(E)|^{-1} \leq 60q^kM|E|^{-1}.\]
This immediately implies the claimed bound $|T| \leq 300q^kM|E|^{-1}$.
\end{proof}

The proof of \cref{conj:ff} is similar to that of \cref{th:fullDimLargeESC}, but we use the following in place of \cref{th:largeESC}.
This was proved in \cite[Theorem 1.4]{LPT}.

\begin{theorem}[\cite{LPT}]\label{th:smallESC}
    Let $E \subset \F_q^d$ with $|E| < 100^{-1}q^{d-1}$. Then,
    \[\#\{y \in \F_q^d: |\pi^y(E)| \leq 10^{-1}|E| \} \leq 8q^{d-1}. \]
\end{theorem}

We are now ready to prove \cref{conj:ff}.

\begin{proof}[Proof of \cref{conj:ff}]
Define 
\[T:=\{y\in \mathbb{F}_q^d\colon |\pi^y(E)|<|E|/50\}.\]
We want to show that $|T|<40q^k$. 

Since $|E| < q^k$, \cref{th:weakProjBound} implies that $|T| < 49^{-1}q|E| < q^{k+1}$.
Hence, if $k < d-1$, then \cref{th:randomOrthogonalProjection} implies that there exists a $(d-k-1)$-dimensional subspace $\Gamma$ such that $|\pi_\Gamma(T)|\ge |T|/5$ and $|\pi_\Gamma(E)|\ge |E|/5$.
If $k=d-1$, we may take $\Gamma = 0$ and obtain the same conclusion, since $\pi_{0}(E) = E$ and $\pi_{0}(T) = T$ by definition.

As the proof of \cref{th:fullDimLargeESC}, we know that 
\[\pi_\Gamma(T)\subseteq \left\lbrace w\in \mathbb{F}_q^{k+1}\colon |\pi^w(\pi_\Gamma(E))|\le |E|/50\right\rbrace \subseteq \left\lbrace w\in \mathbb{F}_q^{k+1}\colon |\pi^w(\pi_\Gamma (E))|\le |\pi_\Gamma(E)|/10\right\rbrace.\]

Since $\pi_\Gamma(E) \subset \F_q^{k+1}$ with $|\pi_\Gamma(E)| \leq |E| \leq 100^{-1}q^k$, \cref{th:smallESC} implies that 
implies that 
\[|\pi_\Gamma(T)| \leq \#\lbrace w\in \mathbb{F}_q^{k+1}\colon |\pi^w(\pi_\Gamma (E))|\le 10^{-1}|\pi_\Gamma(E)|\rbrace \leq 8q^k.\]
Hence, $|T| \leq 5|\pi_\Gamma(T)| \leq 40q^k$.
\end{proof}

\section{Acknowledgements}
P. Bright was supported by the MathWorks fellowship at MIT. B. Lund was partially supported by the Institute for Basic Science (IBS-R029-C1). T. Pham was partially supported by ERC grant ``GeoScape'', no. 882971, under Prof. J\'{a}nos Pach.

B. Lund and T. Pham would like to thank to the VIASM for the hospitality and for the excellent working condition. 
 

 \bibliographystyle{amsplain}

\end{document}